\documentclass[leqno]{amsart}

\numberwithin{equation}{section} 
\usepackage{amsmath,amsfonts,amssymb,amsthm,latexsym}

\usepackage[square,comma,numbers,sort&compress]{natbib}

\theoremstyle{plain}
\newtheorem{theorem}{Theorem}[section]

\newtheorem{lemma}[theorem]{Lemma}

\theoremstyle{definition}
\newtheorem{definition}[theorem]{Definition}

\theoremstyle{remark}
\newtheorem{remark}[theorem]{Remark}

\newcommand{\vect}[1]{\mathbf{#1}}

\newcommand{\vk}{\vect{k}}

\newcommand{\bu}{\vect{u}}
\newcommand{\bv}{\vect{v}}

\newcommand{\field}[1]{\mathbb{#1}}

\newcommand{\nZ}{\field{Z}}

\newcommand{\nR}{\field{R}}

\newcommand{\nT}{\field{T}}

\newcommand{\vphi}{\varphi}
\newcommand{\maps}{\rightarrow}

\newcommand{\cnj}[1]{\overline{#1}}
\newcommand{\pd}[2]{\frac{\partial #1}{\partial #2}}
\newcommand{\fd}[2]{\frac{d #1}{d #2}}

\newcommand{\abs}[1]{\left\lvert#1\right\rvert}

\newcommand{\set}[1]{\left\{#1\right\}}

\newcommand{\pnt}[1]{\left(#1\right)}

\newcommand{\vor}{\boldsymbol{\omega}}

\begin{document}
\title[Effect Of Boundary Conditions]{
Global Regularity vs. Finite-Time Singularities: Some Paradigms on the Effect of Boundary Conditions and Certain Perturbations
}

\date{January 7, 2014}

\author{Adam Larios}
\address[Adam Larios]{Department of Mathematics\\
                Texas A\&M University\\
        College Station, TX 77843, USA}
\email[Adam Larios]{alarios@math.tamu.edu}
\author{Edriss S. Titi}
\address[Edriss S. Titi]{Department of Mathematics, and Department of Mechanical and Aero-space Engineering, University of California, Irvine, Irvine CA 92697-3875, USA. Also The Department of Computer Science and Applied Mathematics, The Weizmann Institute of Science, Rehovot 76100, Israel. Fellow of the Center of Smart Interfaces, Technische Universit\"at Darmstadt, Germany.}
\email[Edriss S. Titi]{etiti@math.uci.edu and edriss.titi@weizmann.ac.il}


\begin{abstract}
In light of the question of finite-time blow-up vs. global well-posedness of solutions to problems involving nonlinear partial differential equations, we provide several cautionary examples which indicate that modifications to the boundary conditions or to the nonlinearity of the equations can effect whether the equations develop finite-time singularities.  In particular, we aim to underscore the idea that in analytical and computational investigations of the blow-up of three-dimensional Euler and Navier-Stokes equations, the boundary conditions may need to be taken into greater account.  We also
examine a perturbation of the nonlinearity by dropping the advection term in the evolution of the derivative of the solutions to the viscous Burgers equation, which leads to the development of singularities not present in the original equation, and indicates that there is a regularizing mechanism in part of the nonlinearity.  This simple analytical example corroborates recent computational observations in the singularity formation of fluid equations.
\end{abstract}


 \maketitle


Burgers Equation, Navier-Stokes Equations, Euler Equations.

 \thispagestyle{empty}

\section{Introduction}\label{sec_Intro_VHJE}

\noindent
A fundamental goal in the study of non-linear initial boundary value problems involving partial differential equations is to determine whether solutions to a given equation develop a singularity in finite time.  Resolving the issue of finite-time blow-up is important, in part because it can have bearing on the physical relevance and validity of the underlying model.  However, determining the answer to this question is notoriously difficult for a wide range of equations; the 3D Navier-Stokes and Euler equations for incompressible fluid flow being perhaps the most well-known examples.  Given that attacking the question directly is so challenging, many researchers have looked for other routes.  One route is to try to simplify or modify the boundary conditions in an attempt to gain evidence for or against the occurrence of finite-time blow-up.  A second route is to modify the equations in some way, and to study the modified equations with the hope of gaining insight into the blow-up of solutions to the  original equations.  In this paper, we will examine several case studies related to such approaches.  A major aim of the present work is to provide examples which demonstrate that one must be extremely cautious in generalizing claims about the blow-up of problems studied in idealized settings to claims about the blow-up of the original problem.  A second aim is to demonstrate a phenomenon which has been observed computationally in the difficult setting of fluid flows in 3D, by means of a simple 1D example, which is amenable to analysis; namely, that a seemingly harmless alteration (from the perspective of enstrophy balance) to the nonlinearity of a problem can cause the formation of a singularity, where no such singularity is present in the unaltered equation.

We will focus on three major cases.  The first case examines the effect of replacing Dirichlet boundary conditions with periodic boundary conditions.  This is often done in both analytical and numerical studies of, e.g.,  the Navier-Stokes and Euler equations.  The original, physical equations come equipped with physical boundary conditions, such as, e.g., Dirichlet boundary conditions in the case of the Navier-Stokes equations.  However, many such studies have tried to search for singularities of the solutions of the equations in the setting of periodic boundary conditions (see, e.g.,  \cite{Kerr_1993,Deng_Hou_Yu_2005,Hou_2009,Hou_Li_2008_Numerical,Hou_Li_2008_Blowup}; in particular, see the surveys \cite{Gibbon_2008,Gibbon_Bustamante_Kerr_2008}, and the references therein).  With this in mind, in section \ref{sec_Periodic_vs_Dirichlet}, we provide an example of an equation which develops a singularity in finite time when Dirichlet boundary conditions are imposed, and yet is globally well-posed in the case of either periodic boundary conditions or the case where the domain is the full space (i.e., in the absence of physical boundaries).  Therefore it may be the case that physical boundary conditions need to be taken into greater consideration in analytic and computational searches for blow-up of the solutions. Indeed, in a recent computational study, the authors of \cite{Hou_Luo_2013_Blowup1D, Luo_Hou_2013_Potentially_Singular} observe the formation of a  finite-time singularity {\it near the boundary} in the 3D Euler equations, of axi-symmetric flow confined in a physical cylinder, subject to no-normal flow  boundary conditions. Notably, a new blow-up criterion for the 3D Euler equations in bounded domains, subject to no-normal flow  boundary conditions, has been established in \cite{Gibbon_Titi}. It is worth stressing that this new criterion does apply for the periodic boundary conditions case or when  the domain is full space, i.e.~in the absence of physical boundaries. For other issues regarding boundary behavior of the Navier-Stokes and Euler equations see the recent surveys \cite{Bardos-Titi1, Bardos-Titi2} and the references therein.

The above discussion is particularly relevant due to the notion of ``boundary
driven'' mechanisms for possible blow-up.  To illustrate how such a mechanism might
work, we give a heuristic scenario in the context of the Navier-Stokes
equations for fluid flow.  It was shown in the celebrated work \cite{Beale_Kato_Majda_1984} that blow-up of the Euler equations occurs if and only if the vorticity becomes infinite (see also \cite{Beale_Kato_Majda_1984,Prodi_1959,Serrin_1962,Constantin_Fefferman_Majda_1996,Constantin_Fefferman_1993,Cao_Titi_2008,Cao_Titi_2011} for additional blow-up criteria).  Infinite vorticity would also cause the Navier-Stokes solutions to become singular.  Now, in the setting of viscous incompressible fluids, physical boundaries are the source of vorticity shedding.  Indeed, near the physical boundary of a fluid, the ``no-slip'' (Dirichlet) boundary conditions can cause the development of boundary layers, where the vorticity is large.
If the viscous diffusion of the fluid velocity is sufficiently small in comparison to the advection, then large magnitudes of the gradient and the vorticity can be propagated from the boundary layer to the interior of the domain by the nonlinear advection term.
The vorticity can then be further intensified by the nonlinear vorticity stretching term, which may thus lead to blow-up of the solution.  Such a physical mechanism does not exist in the periodic setting, nor in the full space $\nR^3$.  It may therefore be illuminating to pay greater attention to the effect of boundary conditions in the search for the blow-up of solutions to the Navier-Stokes and Euler equations.
%
We do not explore these ideas in greater detail as they are only meant to give motivation.  Instead we examine a different, simpler equation in section \ref{sec_Periodic_vs_Dirichlet}, for which we can provide a definite answer.

In section \ref{sec_KS_BC_Pokhozhaev}, we examine the Kuramoto-Sivashinsky equation in a bounded domain with two different types of boundary conditions.  The question of global well-posedness of this equation, when equipped with certain physically relevant  boundary conditions, is still open. Recently, in \cite{Pokhozhaev_2008}, it was shown that, by applying a different (non-physical) set of third-order boundary conditions, a singularity develops in finite time.  In contrast to this, we provide a different set of (also non-physical) third-order boundary conditions for which the equation is globally well-posed.  Therefore, we maintain that it is difficult to obtain information about the blow-up or global well-posedness of an equation by altering its boundary conditions.

We note that such questions relating boundary conditions to blow-up can be highly relevant to applied and computational problems in science.  Indeed, we recall here that such an issue occurred in the study of the planetary geostrophic model used in ocean dynamics. The model is  derived asymptotically by keeping only the hydrostatic balance of the vertical momentum and the  leading order geostrophic balance of the horizonal momentum, where the latter is damped by  the friction with the continental shelf,  while retaining the relevant physical boundary conditions. In \cite{Cao_Titi_Ziane_2004}, it was observed that this model is over determined and hence is ill-posed (it has more boundary conditions than needed for the underlying PDEs).  This observation explains the numerical  instabilities that had been observed near the boundary in simulations of this model.  The resulting oscillations  had proven difficult to eliminate, and were dealt with in \cite{Cao_Titi_Ziane_2004} by adding artificial higher-order diffusion corresponding to the additional boundary conditions in the model.

It is commonly believed that adding  hyper-viscosity  into a numerical scheme enhances the stability of the underlying scheme. In section \ref{Hyperviscosity} we provide in example which questions the validity of this claim. That is, even though the hyper-viscous term enhances the dissipation of small scales, it destroys the maximum principle, which is an essential property for the global stability in certain  physical systems.

Finally, in section \ref{sec_Hou_dropping_the_transport} we consider a certain type of perturbation of the nonlinearity.  In particular, in the context of the Navier-Stokes or Euler equations, by removing the advection term in the vorticity formulation, several recent works  \cite{Hou_Lei_2009,Deng_Hou_Yu_2005,Hou_2009,Hou_Luo_2013_Blowup1D,Luo_Hou_2013_Potentially_Singular} have observed computationally that the solutions of the altered equations seem to blow up in finite time, naming this phenomenon, {\it ``advection depleting singularity"}.  We give an analogous simple example based on a similar alteration of the 1D viscous Burgers equation, and we show analytically that a singularity develops in finite time, which adds credence to the numerical observations of the aforementioned works.  Indeed, since the viscous Burgers equation is globally well-posed, the development of a singularity in the altered model indicates that the removed portion of the nonlinearity has a regularizing effect. However, it is worth stressing that these alterations turn out to be non-local in nature, and in the context of the hydrodynamics equations they translate to modification in the representation of the pressure term.

Many of the results and proofs are not completely new, but, for the sake of being somewhat self-contained, are collected, compared, and contrasted here.  We also aim to state specific, as opposed to general results, whenever doing so simplifies the exposition. The reason for this approach is that our goal is to lay out a simple set of examples and counter-examples for the use of the reader in considering potential mechanisms for singularity formulation or prevention, in particular, in computational studies.

\section{Preliminaries}\label{sec_pre}
\noindent
In this section, we set some notation and recall basic results used below.
We denote by $L^p$, $W^{s,p}$ the usual Lebesgue and Sobolev spaces.  We denote by $C$, $C'$, $C_\Omega$, etc. generic constants which may vary from line to line.

We recall some basic facts about the Laplace operator $\triangle:=\sum_{i=1}^n\partial_{x_i}^2$ in the setting of either periodic or homogeneous Dirichlet boundary conditions (see, e.g., \cite{Evans_2010} for proofs and further discussion).  Recall that the operator $(-\triangle)^{-1}$, subject to the appropriate boundary conditions,  is a positive-definite, self-adjoint, compact operator from $L^2$ into itself, and therefore it has an orthonormal basis of positive eigenfunctions $\set{\vphi_k}_{k=1}^\infty$ (which are also eigenfunctions of  $-\triangle$), corresponding to a sequence of positive eigenvalues. Since the  eigenvalues of $(-\triangle)^{-1}$ can be ordered to be non-increasing,  we can label the eigenvalues of $-\triangle$, which we denote by $\lambda_k$, to be such that $0<\lambda_1\leq\lambda_2\leq\cdots$.

We will pay special attention to the first eigenfunction of $-\triangle$, subject to homogeneous Dirichlet boundary condition,  namely $\vphi_1$, corresponding to $\lambda_1$. We recall Hopf's Lemma, which states that
$-\pd{\vphi_1}{\nu}>0$ on $\partial\Omega$, where $\nu$ is the outward-pointing normal of $\Omega$.    It can also be shown that $\vphi_1$ is strictly positive on $\Omega$.  For proofs of these facts, see, e.g., \cite{Evans_2010}.

We denote the distance function to the boundary by
\begin{align*}
\text{dist}(x,\partial\Omega):=\inf\set{|x-y|:y\in\partial\Omega}.
\end{align*}

Let $\Omega\subset \nR^n$ be a domain which is bounded in at least one direction.  For all $u\in W^{1,p}_0(\Omega)$, $p\geq 1$,
the following Poincar\'e inequality holds
\begin{equation}\label{poincare}
  \|u\|_{L^p}\leq C_\Omega \|\nabla u\|_{L^p},
\end{equation}
with $C = \lambda_1^{-1/2}$ if $p=2$.

We next recall the Gevrey classes of spatially analytic functions.
\begin{definition}
We define the Gevrey classes $G^{s/2}_{\sigma}(\nT^n)$ of spatially analytic functions on the torus $\nT^n:=\nR^n/(2\pi\nZ)^n$,  to be the set of all $u\in L^2(\nT^n)$ such that $\|u\|_{G^{s/2}_{\sigma}(\nT^n)}<\infty$, where
   \begin{align}
      \|u\|_{G^{s/2}_{\sigma}(\nT^n)}:=\pnt{\sum_{\vk\in \nZ^n}|u_\vk|^2(1+|\vk|^2)^s e^{2\sigma(1+|\vk|^2)^{1/2}}}^{1/2},
   \end{align}
   where $u_\vk$ are the Fourier coefficients of $u$, and where $\sigma>0$.
\end{definition}
\noindent Such functions are called Gevrey regular.  Note that formally setting $\sigma=0$, we recover the usual Sobolev spaces $H^s(\nT^n)$.  Furthermore, it can be shown that for $\sigma>0$, $\sigma$ is comparable to the minimal radius of analyticity.

\section{Periodic Vs. Dirichlet Boundary Conditions}\label{sec_Periodic_vs_Dirichlet}

\noindent
Consider the Cauchy problem for the following viscous Hamilton-Jacobi equation,
\begin{subequations}\label{VHJE}
\begin{align}
\label{VHJE_eqn}
u_t-\triangle u&=|\nabla u|^p,  &&\text{ in }\Omega\times(0,T),
\\\label{VHJE_init}
u(0)&=u_0, &&\text{ in }\Omega,
\end{align}
\end{subequations}
\noindent equipped with either periodic boundary conditions or homogeneous Dirichlet boundary conditions.  Many authors have studied the cases $p\in[0,\infty)$ (see, e.g., \cite{Alaa_1996,Amour_Ben-Artzi_1998,Ben-Artzi_Goodman_Levy_2000,Ben-Artzi_Souplet_Weissler_2002,Ben-Artzi_Souplet_Weissler_1999,Gilding_2005,Gilding_Guedda_Kersner_2003}), but in this work, we will focus on the case $p=4$ for simplicity.  In the case $p=2$, is an integrated version of the viscous Burgers equation, and is sometimes referred to as the Kardar-Parisi-Zhang equation, which is used to model the growth and roughening of certain surfaces, as derived in \cite{Kardar_Parisi_Zhang_1986}.  Furthermore, \eqref{VHJE} is an important test equation, since it is one of the simplest examples of a parabolic PDE with non-linear dependence on the gradient.

In the case of periodic boundary conditions, \eqref{VHJE} with $p \ge 2$ is well-posed, globally in time.  However, in the Dirichlet case, and for $p > 2$, a singularity will develop in finite time, for certain  initial data.  We give a relatively simple proof of the well-posedness in the periodic case with $p=4$.  The proof for $p>2$ is given in  \cite{Gilding_Guedda_Kersner_2003}.  For the proof of blow-up in the Dirichlet case, choosing $p=4$ does not appear to make things significantly simpler than allowing $p>2$, so we give the proof for $p>2$.  We follow closely the proof in \cite{Souplet_2002} to show that a singularity occurs in finite time in the Dirichlet case, at least for sufficiently large initial data in the sense given in \eqref{Souplet_Condition}, below.

It is worth noting that the following identity holds for sufficiently smooth functions $u=u(t,x)$:
\begin{align*}
   (\partial_t-\triangle)e^u = (\partial_t u-\triangle u-|\nabla u|^2)e^u,
\end{align*}
Thus, \eqref{VHJE} can be solved explicitly in the case $p=2$, by making the change of variables $v=e^u$ (known as the Cole-Hopf transformation for the Burgers equation), and noting that if $u$ solves \eqref{VHJE}, then $v$ solves the linear heat equation, with the corresponding boundary conditions.


\subsection{Global Well-Posedness in the Periodic Case}
We prove the global existence of solutions to \eqref{VHJE}, for $p=4$ under the assumption of periodic boundary conditions.  We begin by stating a special case of a theorem in \cite{Ferrari_Titi_1998} (which follows ideas from \cite{Foias_Temam_1989}), that gives short-time existence, uniqueness, and regularity.

\begin{theorem}[\cite{Ferrari_Titi_1998}]\label{t:KPZ_p=4_local}
   Let $u_0\in H^s(\nT^n)$, with $s>n/2$, and $\|u_0\|_{H^s(\nT^n)}\leq M_0$ for some $M_0>0$.  Then there exists a $T>0$ depending only upon $M_0$ such that equation \eqref{VHJE}, with $p$ and positive even integer, has a unique solution $u$ on the interval $[0,T)$ with the initial value $u_0$, which satisfies
   $u\in
   C([0,T); H^s(\nT^n))
   \cap
   L^2((0,T); H^2(\nT^n))
   $,
   $\frac{du}{dt}\in L^2((0,T); L^2(\nT^n))$.  Moreover, $u(\cdot,t)\in G^{s/2}_{t}(\nT^n)$ for $t\in[0,T)$.
\end{theorem}

With this theorem in hand, we now state and prove a global existence theorem for \eqref{VHJE} with $p=4$ and $n=1$.  For global well-posedness in the general case, see  \cite{Gilding_Guedda_Kersner_2003,Gilding_2005}.

\begin{theorem}\label{t:KPZ_p=4_global}
   Suppose $u_0\in H^1(\nT)$, and consider \eqref{VHJE} in the one-dimensional case with periodic boundary conditions, and $p=4$.  Then the unique, Gevrey regular solution given by Theorem \ref{t:KPZ_p=4_local} can be extended to an arbitrarily large time interval $[0,T]$.
\end{theorem}

\begin{proof}
First note that, since $p=4$, the right-hand side of \eqref{VHJE} is real analytic in $u_x$, and we have short-time existence and uniqueness (say, on a time interval $[0,T]$) of \eqref{VHJE} under periodic boundary conditions by using, e.g., the Galerkin method.  Furthermore, as shown in  \cite{Ferrari_Titi_1998}, the solution is Gevrey regular in space.  In particular, it has continuous derivatives of all orders.  It remains to show that the solution exists globally in time.  Suppose $[0,T^*)$ is the maximal interval of existence. If $T^*=\infty$ there is nothing to prove. Therefore, we assume by contradiction that $T^*<\infty$.  From the above regularity, we infer in particular, that $u(\cdot, \frac{T^*}{2}) \in H^2(\nT)$.   We use a technique of E. Hopf and G. Stampacchia (cf. \cite{Kinderlehrer_Stampacchia_1980,Temam_1997_IDDS}) to prove a maximum principle for $u_x$.  Write $v := u_x$ and $v_* (\cdot) := u_x(\cdot,\frac{T^*}{2})$.  For any function $f\in H^1$, we use the standard notation $f^+:=\max\set{f,0}$.  It is a standard exercise (see, e.g., \cite{Evans_2010}, section 5.10) to show that $f\in H^1$ implies $f^+\in H^1$.  Taking the derivative of \eqref{VHJE_eqn}, we have
\(
   v_t -v_{xx} = 4v^3 v_x.
\)
Let us denote
\begin{align*}
   \theta(x,t) := v(x,t)-\|v_*\|_{L^\infty}.
\end{align*}
Since $\|v_*\|_{L^\infty}$ is a constant, $\theta_x=v_x$ and $\theta_t=v_t$, so that
\begin{align}\label{VHJE_stamppaccia_eqn}
   \theta_t - \theta_{xx} - 4 v^3\theta_x =0.
\end{align}
Taking the inner product in $L^2$ of \eqref{VHJE_stamppaccia_eqn} with $\theta^+$, we integrate by parts several times and use the fact that $\theta^+\theta = (\theta^+)^2$ to find
\begin{align*}
   \frac{1}{2}\frac{d}{dt}\|\theta^+\|_{L^2}^2+\|\theta^+_x\|_{L^2}^2
   &=
   \int_\nT 4 v^3\theta_x\theta^+\,dx
   =
   \int_\nT 2 v^3((\theta^+)^2)_x\,dx
   \\&=
   -\int_\nT 6 v^2v_x(\theta^+)^2\,dx
   =
   -\int_\nT 2 v^2((\theta^+)^3)_x\,dx
   \\&=
   \int_\nT 4 vv_x(\theta^+)^3\,dx
   =
  \int_\nT  v((\theta^+)^4)_x\,dx
  \\&=
  -\int_\nT  v_x(\theta^+)^4\,dx
  =
  -\int_\nT \frac{1}{5} ((\theta^+)^5)_x\,dx
  =0.
\end{align*}
Thus, integrating in time, for a.e. $t\in[\frac{T^*}{2},T^*)$ we have
\begin{align*}
   \|\theta^+(t)\|_{L^2}^2 \leq \|\theta^+(\tfrac{T^*}{2})\|_{L^2}^2=0.
\end{align*}
Thus, $\theta^+(t)\equiv 0$, and so, $v(x,t)\leq\|v_*\|_{L^\infty}$, for   $t\in[\frac{T^*}{2},T^*)$.  Similarly, one can show that $-v(x,t)\leq\|v_*\|_{L^\infty}$, and thus we have
\begin{align*}
   \|u_x(t)\|_{L^\infty}\leq\|u_x(\tfrac{T^*}{2})\|_{L^\infty},
\end{align*}
for $t\in[\frac{T^*}{2},T^*)$.
Next, taking the inner product of \eqref{VHJE_eqn} with $u$ and using the Lions-Magenes Lemma, we have, for $t\in[\frac{T^*}{2},T^*)$,
\begin{align*}
   \frac{1}{2}\frac{d}{dt}\|u\|_{L^2(\nT)}^2
   + \|u_x\|_{L^2(\nT)}^2
   &=
   \int_\nT (u_x)^4 u\,dx
   \leq
   \|u_x(\tfrac{T^*}{2})\|_{L^\infty}^4\|u\|_{L^1}
   \leq
   C\|v_*\|_{L^\infty}^4\|u\|_{L^2}.
\end{align*}
Integrating the above inequality now yields
\begin{align*}
   \|u(t)\|_{L^2(\nT)}
   \leq
   \pnt{\|u_x(\tfrac{T^*}{2})\|_{L^\infty(\nT)}^4T + \|u_0\|_{L^2(\nT)}^{1/2}}^2<\infty,
\end{align*}
for $t\in[\frac{T^*}{2},T^*)$.  Thus, from the above and  Theorem \ref{t:KPZ_p=4_local} one can extend the solution beyond $T^*$, which leads into a contradiction. Consequently,  $T^*= \infty$.
\end{proof}

\subsection{Finite-Time Blow-Up in the Dirichlet Case}
In this section, we investigate the existence, uniqueness, and the finite-time blow-up, of solutions to \eqref{VHJE}, under the assumption of Dirichlet boundary conditions.  The short time existence and uniqueness of solutions to \eqref{VHJE} can be proven by using, e.g., Duhamel's principle and the Schauder fixed point theorem, see, e.g. \cite{Friedman_1983}.  We state the theorem without proof here.

\begin{theorem}[Short time existence]\label{t:VHJE_Dirichlet_short_time}
Let $\Omega\subset\nR^n$ be a bounded $C^2$ domain.   Suppose $u_0\in C^{1+\alpha}(\Omega)$ for some $\alpha\in(0,1)$ and $u_0\equiv 0$ on $\partial\Omega$.  Then there exists a $T>0$ such that \eqref{VHJE} has a unique solution in $C^{1+\alpha}(\Omega\times[0,T])\cap C(\cnj{\Omega}\times[0,T])$, with $u\equiv 0$ on $\partial\Omega$.
\end{theorem}

In \cite{Alaa_1996} it is shown that \eqref{VHJE}  under homogeneous Dirichlet boundary conditions, cannot have a global solution if $u_0$ is very irregular (namely, if $u_0$ is a positive measure satisfying certain conditions).  In \cite{Souplet_2002} it is shown that for smooth, but sufficiently large initial data (in a sense of \eqref{Souplet_Condition}, below), the solution $u$ blows up in finite
time for $p>2$.  In particular, it is shown that so-called ``gradient blow-up'' occurs; that is, $u$ remains uniformly bounded, but $\limsup_{t\maps T^*}\|\nabla u(t,\cdot)\|_{L^\infty}=\infty$, where $T^*<\infty$ is the maximal existence time for $u$.
The idea is to exploit properties of the first eigenvalue of the negative Laplacian operator, subject to homogeneous Dirichlet boundary conditions.  Let $\lambda_1>0$ be the smallest eigenvalue of $-\triangle$, with homogeneous Dirichlet boundary conditions, and $\vphi_1$ a corresponding eigenfunction, chosen as in section \ref{sec_pre}.
We begin with two lemmas.  The first is used to support the second, and the second is that $\int_\Omega (\vphi_1(x))^{-\alpha}\,dx<\infty$ for $\alpha\in(0,1)$.   This means that we have a certain growth of $\vphi_1$ near the boundary, and will be needed in subsequent calculations.

\begin{lemma}\label{L:dist}
   \label{PhiDist}
   Assume that $\Omega\subset\nR^n$ is a bounded domain with $C^2$ boundary.  Then there exists a constant $C>0$ such that for all $x\in\Omega$,
   $$\vphi_1(x)\geq C\cdot \emph{\text{dist}}(x,\partial\Omega).$$
\end{lemma}
The proof is a fairly straight-forward application of Hopf's Lemma, but we include it here for completeness.
Note that in the one-dimensional case, with $\Omega = (0,\pi)$, we have $\vphi_1(x) = \frac{\pi}{2\pi}\sin(x)$, and the results of Lemmas \ref{L:dist} and \ref{L:alpha} are trivial in this case.

\begin{proof}
     Since $\Omega$ is $C^2$, it satisfies the interior sphere condition.  Therefore for $x\in\Omega$ sufficiently close to $\partial\Omega$, we may write $x=x_0-s\nu$ for some $x_0\in\partial\Omega$, $s>0$ and where $\nu$ is the exterior normal of $\Omega$.  Then $\text{dist}(x,\partial\Omega)=\text{dist}(x,x_0)=s$.  Note that by Hopf's Lemma, $\lim_{s\maps0^+}s^{-1}\vphi_1(x_0-s\nu)>0$.  Since $\partial\Omega$ is $C^2$, and $\Omega$ is bounded, we have sufficient regularity on $\vphi_1$ to conclude that this limit is uniform in $x_0$.  Thus (choosing $x$ closer to $\partial\Omega$ if necessary), there exists a constant $\tilde{C}>0$ such that $s^{-1}\vphi_1(x_0-s\nu)\geq \tilde{C}$, that is, $\vphi_1(x)\geq \tilde{C}\cdot\text{dist}(x,\partial\Omega)$ for all $x$ sufficiently close to the boundary, say within $\epsilon-$neighborhood, for some  $\epsilon\in(0,1)$.  Since $\Omega$ is bounded, the set $E:=\set{x\in\Omega:\text{dist}(x,\partial\Omega)\geq\epsilon}$ is compact. By the elliptic maximum principle, $\vphi_1\geq\tilde{C}\epsilon$ on $E$.  Setting $C=\tilde{C}\epsilon/(1+\text{diam}(\Omega))$, we have $\vphi_1(x)\geq C\cdot\text{dist}(x,\partial\Omega)$, as desired.
\end{proof}
Using this lemma, we next show that $\vphi_1$ satisfies a certain growth condition near its zeros (that is, near the boundary).  This is the main property that is exploited in \cite{Souplet_2002} to show finite-time blow-up.  A crucial lemma, proved in \cite{Souplet_2002}, but also stated (without proof) in \cite{Alaa_1996} and \cite{Fila_Lieberman_1994}, is the following.
\begin{lemma}\label{L:alpha}
   Assume that $\Omega$ is a bounded domain in $\nR^n$ with $C^2$ boundary, and let $\alpha\in(0,1)$.  Then
   $$\int_\Omega\vphi_1^{-\alpha}(x)\,dx<\infty.$$
\end{lemma}
\begin{proof}
   We follow closely the proof in \cite{Souplet_2002}.  We can use a partition of unity to reduce to a local argument.  Since $\Omega$ is bounded with $C^2$ boundary, it can be locally represented as the graph of a $C^2$ function, say $f:U_0\maps(-\epsilon,\epsilon)$, for some $\epsilon>0$, where $x=(x_n',x_n)\in \nR^{n-1}\times\nR$, and
\begin{align*}
U := \set{x\in\nR^n:|x_n'|\leq \epsilon, |x_n|<\epsilon}
,\quad
U_0 := \set{x_n'\in\nR^{n-1}: |x_n'|\leq\epsilon}, \epsilon>0,
\end{align*}
and $f(0)=0$.  Furthermore, we define
\begin{align*}
   \omega := \set{x\in U: x_n<f(x_n')}
   ,\quad
   \Gamma := \set{x\in U: x_n = f(x_n')}.
\end{align*}
By projecting the vector $(\vec{0},f(x_n')-x_n)$ onto the inward-pointing normal of the graph of $f$, it follows that
   \begin{align*}
      \text{dist}(x,\partial\Omega)\geq \frac{f(x_n')-x_n}{\sqrt{1+\|\nabla f\|_{L^\infty}^2}}.
   \end{align*}
Using this and Lemma \ref{L:dist}, we estimate
\begin{align*}
   \int_\omega(\vphi_1(x))^{-\alpha}\,dx
   &\leq
   c^{-\alpha}\int_\omega (\text{dist}(x,\Gamma))^{-\alpha}\,dx
   =
   c^{-\alpha}\int_{U_0}\int_{-\epsilon}^{f(x_n')}(\text{dist}(x,\Gamma))^{-\alpha}\,dx_n\,dx_n'
   \\&\leq
   c^{-\alpha}\pnt{1+\|\nabla f\|_{L^\infty}^2}^{\alpha/2}
   \int_{U_0}\int_{-\epsilon}^{f(x_n')}(f(x_n')-x_n)^{-\alpha}\,dx_n\,dx_n'
   <\infty,
\end{align*}
since $\alpha<1$.
\end{proof}

Next, we seek a lower bound on $\int_\Omega |\nabla u|^p\vphi_1(x)\,dx$.
\begin{lemma}\label{l:poincare_phi_1}
   Let $p>2$, and suppose $\Omega\subset\nR^n$ is a $C^2$ bounded  domain.  Then there exists a constant $C_{\Omega,p}>0$ such that for any function $v\in W^{1,p}_0(\Omega)$,
   \begin{align}\label{poincare_phi_1}
C_{\Omega,p}\abs{\int_\Omega v(t,x)\vphi_1(x)\,dx}^p
\leq
\int_{\Omega}|\nabla v|^p \vphi_1\,dx.
   \end{align}
   Furthermore, no such constant exists for $1\leq p\leq 2$.
\end{lemma}

\begin{proof} For the proof  we follow closely \cite{Souplet_2002}.  Using H\"older's inequality, we have
\begin{align}
\int_{\Omega}|\nabla v|\,dx
&=\int_{\Omega}|\nabla v| \vphi_1^{1/p}\vphi_1^{-1/p}\,dx\notag
\leq\notag
\pnt{\int_{\Omega}|\nabla v|^p\vphi_1\,dx}^{1/p} \pnt{\int_{\Omega}\vphi_1^{-1/(p-1)}\,dx}^{1-1/p}
\\&\leq
\label{holder}
C_{\Omega,p}'\pnt{\int_{\Omega}|\nabla v|^p\vphi_1\,dx}^{1/p},
\end{align}
where, due to Lemma \ref{L:alpha},
\[C_{\Omega,p}':= \pnt{\int_{\Omega}\vphi_1^{-1/(p-1)}\,dx}^{1-1/p}<\infty,\]
since $p>2$ implies $\frac{1}{p-1}\in(0,1)$.

Since $\Omega$ is bounded, we have by Poincar\'e's inequality \eqref{poincare}, that there exists a constant $C_\Omega$ such that
\begin{equation}\label{BC_Paper_poincare}
\int_\Omega|v(x,t)|\,dx
\leq
C_\Omega\int_\Omega|\nabla v(x,t)|\,dx.
\end{equation}
Therefore, \eqref{holder} and \eqref{BC_Paper_poincare} give
\begin{align}
\label{z_est}
\abs{\int_\Omega v(t,x)\vphi_1(x)\,dx}^p
&\leq
\pnt{\|\vphi_1\|_{L^\infty}\int_\Omega |v(t,x)|\,dx}^p
\\&\leq\notag
\pnt{\|\vphi_1\|_{L^\infty}
C_\Omega\int_\Omega|\nabla v(x,t)|\,dx}^p
\\&\leq\notag
\Big(\|\vphi_1\|_{L^\infty} C_\Omega
C_{\Omega,p}'\Big)^p
\int_{\Omega}|\nabla v(x,t)|^p \vphi_1(x)\,dx,
\end{align}
Setting $C_{\Omega,p}=\pnt{\|\vphi_1\|_{L^\infty} C_\Omega
C_{\Omega,p}'}^{-p}$ yields \eqref{poincare_phi_1}.

We next give a counterexample to show that \eqref{poincare_phi_1} cannot hold for all $v\in W^{1,p}(\Omega)$ in the case $p=2$.  Here, for simplicity, we only show the one-dimensional case, with $\Omega=(0,2\pi)$, since similar counterexamples can be constructed in higher dimensional cases based on the one-dimensional case, using the fact that the domain satisfies the interior ball condition and comparing to the distance function, as in the proof of Lemma \ref{L:alpha}.

For $\epsilon\in(0,\pi/4)$,
consider the function $v_\epsilon$ defined on $[0,\pi]$, given by
\[
   v_\epsilon(x):=
   \begin{cases}
      0 &\text{ for } x\in[0,\epsilon^2]\cup [\pi-\epsilon^2,\pi],
      \\
      \log(x/\epsilon^2) &\text{ for } x\in[\epsilon^2,\epsilon],
      \\
      \log(1/\epsilon) &\text{ for } x\in[\epsilon,\pi-\epsilon],
      \\
      \log((\pi-x)/\epsilon^2) &\text{ for } x\in[\pi-\epsilon,\pi-\epsilon^2].
   \end{cases}
\]
We calculate the derivative
\[
   v_\epsilon'(x)=
   \begin{cases}
      0 &\text{ for } x\in(0,\epsilon^2)\cup(\epsilon,\pi-\epsilon)\cup (\pi-\epsilon^2,\pi),
      \\
      1/x &\text{ for } x\in(\epsilon^2,\epsilon),
      \\
      1/(x-\pi) &\text{ for } x\in(\pi-\epsilon,\pi-\epsilon^2).
   \end{cases}
\]
Notice that $v_\epsilon\in W^{1,2}_0((0,\pi))$.  Furthermore, using  the fact that $\sin(x)\geq \frac{2}{\pi}x$ on $[0,2/\pi]$ and that $0< \epsilon< \pi/4$ we have
\begin{align*}
   &\int_0^\pi v_\epsilon(x)\sin(x)\,dx
   =
   2\int_{0}^{\pi/2} v_\epsilon(x)\sin(x)\,dx
   \geq
   \frac{4}{\pi}\int_0^{\pi/2} v_\epsilon(x)x\,dx
   \\\geq&
   \frac{4}{\pi}\int_{\epsilon}^{\pi/2} \log\pnt{\frac{1}{\epsilon}}x\,dx
   =
   \frac{2}{\pi}\log\pnt{\frac{1}{\epsilon}}\pnt{\frac{\pi^2}{4}-\epsilon^2}
   >
   \frac{3\pi}{8}\log\pnt{\frac{1}{\epsilon}},
\end{align*}
Thus,
\begin{align*}
   \pnt{\int_0^\pi v_\epsilon(x)\sin(x)\,dx}^2
   >
   \frac{9\pi^2}{64}\pnt{\log\pnt{\frac{1}{\epsilon}}}^2.
\end{align*}
On the other hand, notice that
\begin{align*}
   \int_0^\pi (v_\epsilon'(x))^2\sin(x)\,dx
   &=2\int_{\epsilon^2}^\epsilon\pnt{\frac{1}{x}}^2\sin(x)\,dx
   \leq
   2\int_{\epsilon^2}^\epsilon\pnt{\frac{1}{x}}^2x\,dx
   =2\log\pnt{\frac{1}{\epsilon}}.
\end{align*}
Taking ratios of the above inequalities, we observe
\begin{align*}
   \frac{\pnt{\int_0^\pi v_\epsilon(x)\sin(x)\,dx}^2}{\int_0^\pi (v_\epsilon'(x))^2\sin(x)\,dx}
   \geq
   \frac{\frac{9\pi^2}{64}\pnt{\log\pnt{\frac{1}{\epsilon}}}^2}{2\log\pnt{\frac{1}{\epsilon}}}
   =\frac{9\pi^2}{128}\log\pnt{\frac{1}{\epsilon}}\maps\infty
\end{align*}
as $\epsilon\maps 0^+$, and therefore no finite number $C>0$ can be chosen to make \eqref{poincare_phi_1} true for all functions $v\in W^{1,2}_0((0,\pi))$.
\end{proof}

\begin{remark}
  The counterexample for $p=2$ we believe to be new.  A counterexample, based on a piecewise linear function, was given in the case $1\leq p<2$, for  $n=1$, in \cite{Bellout_Benachour_Titi_2003}.
\end{remark}

With the above lemmas in hand, we are now ready to prove the main theorem for blow-up of \eqref{VHJE} under homogeneous Dirichlet boundary conditions.  As mentioned earlier, the proof is very similar to the one given in \cite{Souplet_2002}, where it is also given in greater generality.  The proof is given here for the sake of completeness.


\begin{theorem}\label{T:blowup}
Let $p>2$ and suppose $u_0\in C^2(\Omega)\cap L^\infty(\Omega)$.  There exists $K>0$, given by equation \eqref{Souplet_Condition_Constant} below, such that if
\begin{align}\label{Souplet_Condition}
\int_\Omega u_0(x)\vphi_1(x)\,dx \geq K,
\end{align}
then any solution to \eqref{VHJE}, taken with homogeneous Dirichlet boundary conditions and initial data $u_0$, develops a singularity in finite time.
\end{theorem}
%

\begin{proof}
By Theorem \ref{t:VHJE_Dirichlet_short_time}, we know that there exists a time $T>0$ and a unique solution $u\in C^{1+\alpha}([0,T]\times\Omega)\cap C([0,T]\times\cnj{\Omega})$ to \eqref{VHJE} satisfying $u\equiv0$ on $\partial\Omega$.
Let $T^*>0$  be the maximal  time of existence of the solution to \eqref{VHJE}. If   $T^*< \infty$ then there is nothing to prove. Therefore, we assume by contradiction that $T^* =\infty$. Following \cite{Souplet_2002}, let
\begin{align}
z(t)=\int_\Omega u(t,x)\vphi_1(x)\,dx.                                                                                                                        \end{align}
The use of $z(t)$ will allow us to use standard non-existence results for ODEs, exploiting properties of $\vphi_1$.
Integrating by parts, we calculate in $(0,T^*)$,
\begin{align}
z'(t)+\lambda_1z(t)
&=\label{z1}
\int_\Omega u_t(t,x)\vphi_1(x)\,dx
-\int_\Omega u(t,x)\triangle \vphi_1(x)\,dx
\\&=\notag
\int_\Omega \big(u_t(t,x)
-\triangle u(t,x)\big)\vphi_1(x)\,dx
=\int_\Omega |\nabla u|^p\vphi_1(x)\,dx.
\end{align}
Applying Lemma \ref{l:poincare_phi_1} to \eqref{z1} gives
\begin{equation*}
z'(t)+\lambda_1z(t)
\geq
C_{\Omega,p}(z(t))^p.
\end{equation*}
Now, if
\begin{align}\label{Souplet_Condition_Constant}
   z(0)\geq\pnt{2\lambda_1/C_{\Omega,p}}^{1/(p-1)} =: K,
\end{align}
then the above estimate implies that $z'(t)\geq 0$ for a short interval of time, and thus $z(t)\geq K$ for all $t\geq0$. Let $y(t):=e^{\lambda_1 t}z(t)$.  Notice that $y(0)=z(0)$.  We then have
\begin{align*}
   y'(t)\geq C_{\Omega,p}e^{\lambda_1 (1-p)t}(y(t))^p
\end{align*}
Integrating, we obtain
\begin{align*}
   (y(0))^{1-p}-(y(t))^{1-p}
   \geq
   C_{\Omega,p}\lambda_1^{-1}(1-e^{\lambda_1 (1-p)t}).
\end{align*}
Thus,
\begin{align}\label{KPZ_y_lower_bound}
   (y(t))^{p-1}
   \geq
   \pnt{(z(0))^{1-p}-C_{\Omega,p}\lambda_1^{-1}(1-e^{\lambda_1 (1-p)t})}^{-1}.
\end{align}
Now, since $z(0)\ge K$, then the right-hand side of \eqref{KPZ_y_lower_bound} become infinite at finite time $t = T^{**}$, where $e^{\lambda_1 (1-p)T^{**}}=1/2$. Hence $T^* \leq T^{**}$, which contradicts the assumption that $T^* =\infty$. In particular, we have shown that a singularity of $u$ develops in finite time.
\end{proof}

\begin{remark}
Regarding the previous theorem, note that, by the maximum principle, since $u\equiv 0$ on $\partial\Omega$,  the extreme values must occur at the initial time, so that
\begin{equation}
 \sup_{(x,t)\in\Omega\times[0,T^*)}|u(x,t)|
\leq
\|u_0\|_{L^\infty}
<\infty.
\end{equation}
Thus, since  the solution ceases to exist after finite time, it must do so in a norm other than $L^\infty$.
\end{remark}

\begin{remark}\label{r:p.lt.2}
   In the previous theorem, the condition $p>2$ is sharp, since \eqref{VHJE} has global existence in the Dirichlet case when $p\leq 2$, (see, e.g., \cite{Ferrari_Titi_1998,Gilding_Guedda_Kersner_2003,Gilding_2005}).  Furthermore, one can see that the reason why the proof of Lemma \ref{T:blowup} fails is because inequality \eqref{poincare_phi_1} fails in this case.
\end{remark}



Let us conclude  by remarking that in this section, we have seen that in the Dirichlet case, if one chooses smooth initial data, say $u_0\in C^{1+\alpha}(\Omega)$ such that $\int_\Omega u_0\vphi_1\,dx$ is sufficiently large, then the solution to \eqref{VHJE} will blow up in finite time.  However, in the case of periodic boundary conditions, specifying that $u_0\in H^1(\nT)$ (in fact, as shown in \cite{Gilding_Guedda_Kersner_2003,Gilding_2005}, one only needs $u_0\in C(\nT^n)$), the solution to \eqref{VHJE} will exist globally in time.  Similar results hold in full space, if one assumes, e.g., that $u_0\in C(\nR^n)\cap L^\infty(\nR^n)$ (see, e.g., \cite{Gilding_Guedda_Kersner_2003,Gilding_2005}).  Thus, it may be the case that computational and analytic searches for blow-up in more complicated situations (e.g., the Navier-Stokes and Euler equations) might not provide evidence for blow-up, due to the fact that full-space or periodic boundary conditions essentially neglect any effects of the boundary, even if blow-up does occur in the case of physical boundary conditions.

\section{Inferences From Adjusting Boundary Conditions}\label{sec_KS_BC_Pokhozhaev}

In this section, we consider the claims that can be made with reference to certain boundary conditions.  Previously, we saw that changing from periodic boundary conditions to physical (e.g., Dirichlet-like) boundary conditions could determine whether or not a solution is globally well-posed.  Next, we will show an example in which a problem, given by the Kuramoto-Sivashinsky equations, is, in the one-dimensional case, globally well-posed under periodic, full-space, or Neumann-like boundary conditions, but loses its global regularity, in any dimension,  under another set of boundary conditions given by \eqref{KS_BC_Pokhozhaev}, below.  That is to say, in certain settings, one may engineer certain (possibly non-physical) boundary conditions to force finite-time blow-up of to occur.

We consider the Kuramoto-Sivashinsky equations, given by
\begin{subequations}\label{KS}
\begin{align}\label{KS_eqn}
   u_t + \triangle^2u+\triangle u +\tfrac{1}{2}|\nabla u|^2&=0 && \text{ in }\Omega\times(0,T),
   \\\label{KS_init}
   u(x,0) &=u_0(x) && \text{ in }\Omega.
\end{align}
\end{subequations}
This form of the Kuramoto-Sivashinsky equations is sometimes called the integrated version of the Kuramoto-Sivashinsky equations.
Here, we consider $\Omega\subset\nR^n$ to be a smooth domain.  We will discuss several variations on the boundary conditions below, as these are the major focus of this section.   Currently, even in the one-dimensional case, the question of global existence of solutions to \eqref{KS} under the physical Dirichlet-like boundary conditions
\begin{align}\label{KS_BC_Physical}
   u \equiv \triangle u \equiv 0 \text{ on }\partial\Omega,
\end{align}
is still open.   Moreover, for $n\geq2$, the question of global well-posedness of \eqref{KS} in the periodic case, or in $\nR^n$ is also an open challenging question.

As it turns out, dealing with the spatial average of the solution can be the main obstacle in showing global regularity for \eqref{KS}, \eqref{KS_diff}, and to avoid this issue, many authors set $\bv = \nabla u$, and consider instead the differentiated version of \eqref{KS}, i.e., the system
\begin{subequations}\label{KS_diff}
\begin{align}\label{KS_diff_eqns}
   \bv_t + \triangle^2\bv+\triangle \bv +(\bv\cdot\nabla) \bv&= 0&& \text{ in }\Omega\times(0,T),
   \\\label{KS_diff_init}
   \bv(x,0) &=\bv_0(x):=\nabla u_0 (x) && \text{ in }\Omega.
\end{align}
\end{subequations}
It is well-known that in the one-dimensional case, with either periodic
($\Omega=\nT:=\nR/\nZ$) or full-space ($\Omega=\nR$) boundary conditions,
\eqref{KS_diff} is globally well-posed, and in the periodic case, has a finite-dimensional global attractor and an inertial manifold (see, e.g., \cite{Tadmor_1986,Robinson_2001,Temam_1997_IDDS,Ilyashenko_1992,Collet_Eckmann_Epstein_Stubbe_1993_Attractor,Goodman_1994,Foias_Nicolaenko_Sell_Temam_1985,Constantin_Foias_Nicolaenko_Temam_1989,Constantin_Foias_Nicolaenko_Temam_1989_IM_Book,Foias_Sell_Titi_1989,Foias_Sell_Temam_1985} and the references therein).  It was shown in \cite{Cao_Titi_2006_KSE} that the only steady-state solutions to \eqref{KS} in either $\nR^n$ or $\nT^n$, $n=1,2$, are constant functions.  The question of the global well-posedness of \eqref{KS} for $n\geq 2$ in the periodic case, or $\nR^n$ is still open.  There have been partial results in bounded domains in dimension $n\geq2$, assuming special geometries.  For instance, global well-posedness for \eqref{KS} in dimension $n=2,3$ was shown in \cite{Bellout_Benachour_Titi_2003} for the case of radially symmetric initial data, in an annular domain $\Omega=\set{x : 0<r<|x|<R}$, where $r,R$ are fixed positive numbers, and the Neumann boundary conditions $\partial_r u = \partial_r\triangle u =0$ on $\partial\Omega$ are imposed.  However, the general case is currently an outstanding open problem.

One can also consider a generalization of \eqref{KS_eqn}, namely
\begin{align*}
   u_t + \triangle^2u+\triangle u +\tfrac{1}{2}|\nabla u|^p=0,
\end{align*}
for some $p\geq0$.  This equation was considered in \cite{Bellout_Benachour_Titi_2003}, where it was shown that when $p>2$, under the boundary conditions \eqref{KS_BC_Physical}, a singularity develops in finite-time, provided that the initial data is sufficiently large in a certain sense, similar to \eqref{Souplet_Condition}.  (In fact, in \cite{Bellout_Benachour_Titi_2003}, the authors proved an even stronger result, as they did not need the destabilizing term $\triangle u$.) The result and proof are similar in character to that of Theorem \ref{T:blowup}, although care needs to be taken due to the fact that one no longer has a maximum principle.

Recently, in \cite{Pokhozhaev_2008}, it has been shown that, for any dimension, a finite-time singularity will develop in solutions to \eqref{KS} for a certain class of initial conditions, if one imposes the boundary conditions
\begin{align}\label{KS_BC_Pokhozhaev}
   u =0, \qquad \pd{}{\nu}(u+\triangle u) =0  \text{ on } \partial\Omega\times(0,T).
\end{align}
The blow-up can be shown by the following calculation, which occurs in \cite{Pokhozhaev_2008} (see also \cite{Galaktionov_Mitidieri_Pokhozhaev_2008}).  Integrating equation \eqref{KS_eqn} in space and using the divergence theorem with boundary conditions \eqref{KS_BC_Pokhozhaev}, the Poincar\'e inequality, and the Cauchy-Schwarz inequality, we find
\begin{align*}
   \fd{}{t}\int_\Omega u\,dx
   &=
   \int_\Omega (-\triangle^2 u-\triangle u + \tfrac{1}{2}|\nabla u|^2)\,dx
   =
   \frac{1}{2}\int_\Omega |\nabla u|^2\,dx
   \\&\geq
   \frac{\lambda_1}{2}\int_\Omega |u|^2\,dx
   \geq
   \frac{\lambda_1}{2|\Omega|}\pnt{\int_\Omega u\,dx}^2.
\end{align*}
Gr\"onwall's inequality then yields
\begin{align*}
   \int_\Omega u(t,x)\,dx
   \geq
   \pnt{1-\frac{\lambda_1t}{2|\Omega|}\int_\Omega u_0(x)\,dx}^{-1}\int_\Omega u_0(x)\,dx.
\end{align*}
Thus, if we choose the initial data such that
\begin{align*}
   \int_\Omega u_0\,dx>0,
\end{align*}
the solution will blow up at least by time $T^*=\frac{2|\Omega|}{\lambda_1}\pnt{\int_\Omega u_0(x)\,dx}^{-1}<\infty$.

Thus, we have seen that one can impose boundary conditions, namely conditions  \eqref{KS_BC_Pokhozhaev}, to cause the solution of \eqref{KS} to blow up for certain initial data.  However, we observe again that  problem \eqref{KS} under boundary conditions \eqref{KS_BC_Physical} still remains open.   Furthermore, we show below that if one imposes somewhat looser boundary conditions than \eqref{KS_BC_Pokhozhaev}, one can show that the solution does not blow up in finite time, at least in the one-dimensional case.

\begin{theorem}
   Consider the one-dimensional version of \eqref{KS} on the domain $\Omega=(0,L)\subset\nR$, with the boundary conditions
   \begin{align}\label{KS_Titi_BC}
      u_x(0)=u_{xxx}(0)=u_x(L)=u_{xxx}(L)=0.
   \end{align}
   Given $T>0$ and $u_0\in H^1((0,L))$ satisfying \eqref{KS_Titi_BC}, there exists a solution $u$ of \eqref{KS} such that $u\in L^\infty([0,T];L^2((0,L)))\cap L^2([0,T];H^3((0,L)))$ to \eqref{KS}.  Furthermore, this solution is unique, and is Gevrey regular in space for $t>0$.
\end{theorem}


\begin{proof}
   We give only a formal existence proof here, but we remark that the proofs can be made rigorous by using, e.g., the Galerkin procedure.  First, we notice that one can show the short-time existence by using the Galerkin procedure based on the eigenfunctions of $-\partial_{xx}$ with Neumann boundary conditions $u_x=0$ for $x=0,L$ (i.e., functions of the form $\cos(\pi k x /L)$.  The proof of this is similar to standard proofs in periodic boundary conditions.  We refer to   \cite{Robinson_2001,Temam_1997_IDDS} for a demonstration of this method, and also a proof of uniqueness.
   Furthermore, as in \cite{Foias_Temam_1989,Ferrari_Titi_1998,Liu_1991_KSE_Gevrey,Collet_Eckmann_Epstein_Stubbe_1993_Analyticity}, one can show that the solutions are Gevrey regular (analytic) in space for $t>0$.
   It remains to show that the solution remains bounded for all time.

   Formally taking the $L^2$ inner-product of \eqref{KS_eqn} with $-u_{xx}$ and integrating by parts, we find
   \begin{align*}
      \frac{1}{2}\frac{d}{dt}\|u_x\|_{L^2}^2
     +
     \|u_{xxx}\|_{L^2}^2
     =
     \|u_{xx}\|_{L^2}^2
     +
     \int_0^L u_x^2u_{xx}\,dx.
   \end{align*}
Notice that $\int_0^L u_x^2u_{xx}\,dx = \int_0^L \frac{1}{3}\partial_x(u_x)^3\,dx
=0$.  Furthermore, integrating by parts and using the Cauchy-Schwarz inequality, we have
\begin{align}\label{KS_interpolation}
\|u_{xx}\|_{L^2}^2 = -\int_0^L u_xu_{xxx}\,dx\leq \|u_x\|_{L^2}\|u_{xxx}\|_{L^2}.\end{align}
  Young's inequality gives $\|u_x\|_{L^2}\|u_{xxx}\|_{L^2}\leq \frac{1}{2}\|u_x\|_{L^2}^2+\frac{1}{2}\|u_{xxx}\|_{L^2}^2$.  Combining these estimates, we have
\begin{align}\label{KS_est_u_x}
\frac{d}{dt}\|u_x\|_{L^2}^2
     +
     \|u_{xxx}\|_{L^2}^2 \leq \|u_x\|_{L^2}^2
\end{align}
Dropping the term $\|u_{xxx}\|_{L^2}^2$ and using Gr\"onwall's inequality, we find
\begin{align*}
   \|u_x(t)\|_{L^2}^2\leq e^t\|u_x(0)\|_{L^2}^2,
\end{align*}
so that $u_x\in L^\infty([0,T];L^2(0,L))$.  Moreover, integrating \eqref{KS_est_u_x} on $[0,t]$, $t\leq T$, we have
\begin{align*}
   \|u_x(t)\|_{L^2}^2
     +
     \int_0^t\|u_{xxx}(s)\|_{L^2}^2\,ds
     &\leq
     \|u_x(0)\|_{L^2}^2
     +
     \int_0^t\|u_x\|_{L^2}^2\,ds
     \\&\leq
     \|u_x(0)\|_{L^2}^2
     +
     \int_0^t e^s\|u_x(0)\|_{L^2}^2\,ds
     =
     e^t\|u_x(0)\|_{L^2}^2.
\end{align*}
Thus, $u_{xxx}\in L^2([0,T];L^2(0,L))$.  Using \eqref{KS_interpolation} and the above estimates, we find that $u_{xx}\in L^4([0,T];L^2)\subset L^2([0,T],L^2)$.  In order to show  $u\in  L^2([0,T],H^3)$, it remains to prove $u\in  L^2([0,T],L^2)$.  (Note that this does not follow directly from the Poincar\'e inequality, due to the boundary conditions \eqref{KS_Titi_BC}.)

Integrating \eqref{KS_eqn} over $[0,L]$, we find
\begin{align}\label{KS_u_av_eqn}
   \frac{d}{dt}\int_0^L u\,dx
   =
   -\int_0^L u_{xxxx}\,dx-\int_0^L u_{xx}\,dx
   -\int_0^L u_{x}^2\,dx
   = -\|u_{x}\|_{L^2}^2.
\end{align}
Denoting $\cnj{\vphi}:=\int_0^L\vphi(x)\,dx$, and integrating \eqref{KS_u_av_eqn} on $[0,t]$, $t\leq T$, we have
\(
   \cnj{u}(t) = \cnj{u}(0) - \int_0^t \|u_{x}(s)\|_{L^2}^2\,ds,
\)
so that
\begin{align*}
   |\cnj{u}(t)|\leq |\cnj{u}(0)|+ \int_0^t \|u_{x}(s)\|_{L^2}^2\,ds
   \leq
   |\cnj{u}(0)|+  (e^t-1)\|u_x(0)\|_{L^2}^2.
\end{align*}
Thus, $\cnj{u}\in L^\infty([0,T])$.  Now, by the Poincar\'e inequality, we have
\begin{align*}
   \|u\|_{L^2(0,L)}-L^{1/2}|\cnj{u}|
   =
   \|u\|_{L^2(0,L)}-\|\cnj{u}\|_{L^2(0,L)}
   \leq
   \|u-\cnj{u}\|_{L^2(0,L)}
   \leq
   C\|u_x\|_{L^2(0,L)}.
\end{align*}
Combining the above estimates, we find
\begin{align*}
   \|u(t)\|_{L^2(0,L)}
   &\leq
   L^{1/2}|\cnj{u}(t)|+C\|u_x(t)\|_{L^2(0,L)}
   \\&\leq
   L^{1/2}\pnt{|\cnj{u}(0)|+  (e^t-1)\|u_x(0)\|_{L^2}^2}+Ce^{t/2}\|u_x(0)\|_{L^2}.
\end{align*}
Thus, we have $u\in L^\infty([0,T];L^2)$, and therefore from the above estimates, it follows that $u\in L^2([0,T];H^3)$, where the bound depends only upon $|\cnj{u}(0)|,\|u_x(0)\|_{L^2},T$, and $L$.  Thus, the solution can be extended globally in time.  In particular, $u_x\in L^2((0,T),H^3)\subset L^2((0,T),L^\infty)$.
\end{proof}

\begin{remark}
   Observe that in the previous proof, one can see that $v:=u_x$ satisfies \eqref{KS_diff} in the one-dimensional case, with $v(0)=v(L)=v_{xx}(0)=v_{xx}(L)=0$.  Furthermore, notice that, by extending $v$ as an odd function on $[-L,L]$, this is equivalent to the case with periodic boundary conditions on $[-L,L]$, where the functions are restricted to be odd functions, where it is well-known that one has global well-posedness and Gevrey regularity, as studied, e.g., in \cite{Tadmor_1986,Robinson_2001,Temam_1997_IDDS,Ilyashenko_1992} and the references therein.
\end{remark}

\section{Is Hyper-viscosity Stabilizing?}\label{Hyperviscosity}

In many numerical simulations, especially for geophysical flows, a hyper-viscosity term of the form $(-\triangle)^\alpha$, with $\alpha >1$, is used to stabilize the underlying numerical scheme. In the presence of physical boundaries, such as in ocean dynamics models, these artificial hyper-viscosity operators require  additional {\it non-physical} boundary conditions. Even if we set aside this issue with the artificial boundary conditions, we are still faced with the question: is hyper-viscosity is always a stabilizing mechanism? To make our point we consider, for example, the 2D and 3D differentiated form of the viscous Burgers equations:
\begin{equation} \label{vBurgers}
\frac{\partial u}{\partial t} - \nu \triangle u + u \cdot \nabla u =0,
\end{equation}
subject to periodic boundary conditions. On the one hand, and as it was observed in the section \ref{sec_Periodic_vs_Dirichlet}, system \eqref{vBurgers} is globally well-posed for initial data, thanks to the maximum principle, namely, $\|u(\cdot,t)\|_{L^\infty} \leq \|u_0\|_{L^\infty}$, for all $t\ge 0$ (see, e.g., \cite{Ladyzhenskaya_1968}). On the other hand, if one adds a hyper-viscosity term to \eqref{vBurgers}, and consider instead
\begin{equation} \label{hvBurgers}
\frac{\partial u}{\partial t} + \kappa (-\triangle)^{\alpha} u - \nu \triangle u + u \cdot \nabla u =0, \quad \mbox{with} \, \, \alpha >1,
\end{equation}
subject to periodic boundary conditions, nothing is known about the global regularity of \eqref{hvBurgers}, for large initial data, even in the two-dimensional case. This is because we lose the maximum principle in the hyper-viscous case, which is the only global \textit{a priori} bound available for \eqref{vBurgers}.  Indeed, it would be interesting if one could show that \eqref{vBurgers} develops a finite-time singularity, while \eqref{vBurgers} is globally well-posed. In particular, any global regularity result concerning \eqref{vBurgers} will shed light on the question of global regularity for the 2D and 3D Kuramoto-Sivashinsky equation (see, e.g., \cite{Bellout_Benachour_Titi_2003, Cao_Titi_2006_KSE} for further discussion of this problem). We observe that in the 1D case, global regularity can be established by standard energy methods (see, e.g., \cite{Collet_Eckmann_Epstein_Stubbe_1993_Attractor, Constantin_Foias_Nicolaenko_Temam_1989_IM_Book, Goodman_1994,  Tadmor_1986, Temam_1997_IDDS} and references therein).

\section{Singularity Formation by Altering the Nonlinearity}\label{sec_Hou_dropping_the_transport}
\noindent
In this section, we remove the advection term in the evolution of the derivative of the  Burgers equation and show that the resulting equation blows up in finite time.  The alteration  in nonlinearity, and the resulting blow-up phenomenon, are analogous to
a phenomenon observed computationally in \cite{Hou_Lei_2009,Deng_Hou_Yu_2005,Hou_2009} for much more complicated equations governing fluids.   Namely, it was noticed in simulations that, by removing the advection term in the vorticity evolution of the 3D axi-symmetric Navier-Stokes or Euler equations, one can seeming cause these equations to blow up in finite time via a certain mechanism.  However, when the nonlinearity is restored to its original form, the mechanism seems to disappear, and therefore, it is claimed that it is reasonable to expect that the advection is depleting the singularity.  Our purpose in this section has two folds. First,  to show analytically that this phenomenon does indeed occur, at least  in simpler setting of the 1D viscous Burgers equation. The second is to shed more light on this mechanism and to stress that it is not precisely the advection term that depletes singularity; but it is rather a non-local alteration of the nonlinear and the pressure terms that cause this effect, and that, sometimes, such a nonlocal alteration might cause the opposite effect.  In other words, we observe that this alteration of the advection term, in the evolution of the derivative in our example, and in the evolution of the vorticity in \cite{Hou_Lei_2009,Deng_Hou_Yu_2005,Hou_2009}, are in fact non-local in nature and are not as naive as they might seem. This is because any non-local change in the nonlinearity of the hydrodynamic equations is in effect an alteration in the representation of the pressure term, which is the major obstacle in the study of three-dimensional hydrodynamic equations.

To illustrate the observation made in \cite{Hou_Lei_2009,Deng_Hou_Yu_2005,Hou_2009}, we consider the unforced 3D Navier-Stokes equations for incompressible flow, namely
\begin{subequations}\label{BC_Paper_NSE}
\begin{align}
\label{BC_Paper_NSE_mom}
    -\partial_t\bu + (\bu\cdot\nabla)\bu &= -\nabla p +\nu\triangle\bu,
    \\
    \label{BC_Paper_NSE_div}
    \nabla\cdot\bu &=0,
\end{align}
\end{subequations}
in the whole space $\nR^3$, and with a given initial condition.  Here
$\bu=\bu(x_1,x_2,x_3,t)$ is the vector-valued velocity of a fluid, and
$p=p(x_1,x_2,x_3,t)$ is the pressure.
Let us define $\vor:=\nabla\times\bu$, which is known as the vorticity of the fluid.  Taking the curl of \eqref{BC_Paper_NSE_mom} and using \eqref{BC_Paper_NSE_div}, we obtain the well-known vorticity equation namely,
\begin{align}\label{NSE_vor}
   \partial_t\vor +(\bu\cdot\nabla)\vor= (\vor\cdot\nabla)\bu+\nu\triangle\vor.
\end{align}
Recently, in \cite{Hou_Lei_2009}, a reformulation of \eqref{NSE_vor} was given in the axi-symmetric case.  Furthermore, it was suggested in \cite{Hou_Lei_2009,Hou_2009}, based on numerical simulations, that in this new formulation, the analogue of the advection term  $(\bu\cdot\nabla)\vor$ may prevent the blow-up of solutions.  Specifically, it is suggested that this term may be responsible for depleting the singularity.  It is worth noting that a similar phenomenon is also conjectured to occur in a generalization of the Constantin-Lax-Majda equation; see, e.g., \cite{Okamoto_Sakajo_Wunsch_2008} and the references therein.

In comparison to the above remarks about the Navier-Stokes equations, we consider the one-dimensional viscous Burgers equation with Neumann boundary conditions on the interval $(0,\pi)$:
\begin{align}\label{viscous_burgers}
   u_t + uu_x  &=\nu u_{xx},
   \qquad
   u_x(0,t)=u_x(\pi,t) = 0,
   \qquad
   u(x,0) = u_0(x),
\end{align}
with viscosity $\nu>0$.   In \cite{Cao_Titi_2002_Burgers_Neumann}, it was proven that, for $u_0\in C(\cnj{\Omega})$, there exists a unique (global) solution $u$ to \eqref{viscous_burgers} satisfying $u\in L^2_{\text{loc}}((0,T],H^3)\cap C((0,T],H^2)\cap L^2((0,T],H^1)\cap C([0,T],L^2)$.  If $u_0\in H^1(\Omega)\cap C(\cnj{\Omega})$, one can additionally show  that $u\in C([0,T],H^1)$.  (See also \cite{Protter_Weinberger_1984,Ladyzhenskaya_1968} for classical results on Burgers' equation.)

Differentiating the equation in \eqref{viscous_burgers} with respect to $x$ and denoting $\omega:=u_x$, we obtain
\begin{align}\label{Burgers_x}
    \omega_t + u\omega_x &= \nu\omega_{xx}-\omega^2.
\end{align}
We show that removing the advection term $u\omega_x$ from \eqref{Burgers_x} allows for solutions which develop a singularity in finite time.  We note that this phenomenon was also pointed out, without proof, in \cite{Okamoto_Sakajo_Wunsch_2008_Ohkitani_Review}, which is a review of  \cite{Okamoto_Sakajo_Wunsch_2008}.  Consider the problem
\begin{align}\label{Burgers_no_transport}
    \omega_t  &= \nu\omega_{xx}-\omega^2,
        \qquad
    \omega\big|_{\partial\Omega}=0,
    \qquad
   \omega(x,0) = \omega_0(x).
\end{align}
A generalized version of the equation in \eqref{Burgers_no_transport} was studied in \cite{Fujita_1966} (see also \cite{Ball_1977,Quittner_Souplet_2007}).
Since \eqref{Burgers_no_transport} is the viscous (PDE) version of the Riccati equation $\dot{y}=y^2$, it is not surprising that it can develop a singularity in finite time.  Indeed, in \cite{Quittner_Souplet_2007} it was proven that for  initial data which is everywhere positive, solutions must develop a singularity in finite time.
%
Here, for the sake of completeness, we give a different proof, which shows that blow-up may also occur for initial data which is not everywhere positive.

\begin{theorem}
   There exists an $M >0$ such that if the initial data $\omega_0\in C([0,\pi])$ satisfies
   \begin{align*}
      \int_0^\pi\omega_0(x)\sin(x)\,dx < -M,
   \end{align*}
then the corresponding solution $\omega$ to \eqref{Burgers_no_transport} blows up in finite time.  More precisely, there exists a time $T^*\in(0,\infty)$ such that
\begin{align}\label{blow_up_omega}
   \lim_{t\maps {T^*}^-}\int_0^\pi \omega(x,t)\sin(x)\,dx =-\infty.
\end{align}
\end{theorem}


\begin{proof}
We proceed somewhat formally, as we only wish to illustrate the main ideas. For notational simplicity, $\vphi(x) = \sin(x)$.
Taking the inner product of \eqref{Burgers_no_transport} with $\vphi$ and integrating by parts twice gives
\begin{align}\label{Ricati_parts}
   \frac{d}{dt}\int_0^{\pi} \omega \vphi\,dx = \nu\int_0^{\pi}\omega\vphi_{xx}\,dx -\int_0^{\pi} \omega^2 \vphi\,dx.
\end{align}
By the Cauchy-Schwarz inequality,
\begin{align*}
   \pnt{\int_0^{\pi} \omega \vphi\,dx}^2
   &=
   \pnt{ \int_0^{\pi} \omega \vphi^{\frac{1}{2}}\vphi^{\frac{1}{2}}\,dx}^2
   \leq
   \pnt{\int_0^{\pi} \omega^2 \vphi\,dx}\pnt{ \int_0^{\pi} \vphi\,dx}
    =
    2\int_0^{\pi} \omega^2 \vphi\,dx.
\end{align*}
Furthermore,
\begin{align*}
   \int_0^{\pi}\omega\vphi_{xx}\,dx
   &=
   \int_0^{\pi}\omega\vphi^{\frac{1}{2}}\vphi_{xx}\vphi^{-\frac{1}{2}}\,dx
   \\&\leq
    \pnt{\int_0^{\pi}\omega^2\vphi\,dx}^{\frac{1}{2}}
    \pnt{\int_0^{\pi}\vphi_{xx}^2\vphi^{-1}\,dx}^{\frac{1}{2}}
  =
 \sqrt{2}\pnt{\int_0^{\pi}\omega^2\vphi\,dx}^{\frac{1}{2}}.
\end{align*}
Using these estimates in \eqref{Ricati_parts} along with Young's inequality yields
\begin{align*}\label{Ricati_subs}
   \frac{d}{dt}\int_0^{\pi} \omega \vphi\,dx
   &\leq
    \nu\sqrt{2}\pnt{\int_0^{\pi}\omega^2\vphi\,dx}^{\frac{1}{2}}
   -\int_0^{\pi} \omega^2 \vphi\,dx
    \\&\leq
    \nu^2-\frac{1}{2}\int_0^{\pi}\omega^2\vphi\,dx
    \leq
    \nu^2-\frac{1}{4}\pnt{\int_0^{\pi} \omega \vphi\,dx}^2.
\end{align*}
Setting $y(t) := \int_0^{\pi} \omega \vphi\,dx$,
we have $\dot{y}\leq \nu^2-y^2/4$.  Choosing $y(0)< - \sqrt{8}\nu$,
we have by continuity that $y(t)< - \sqrt{8}\nu$ for all $t\in[0,\delta]$, for some $\delta >0$.  Thus, $\dot{y}< - \frac{1}{8} y^2$ on $[0,\delta]$, so $y$ is decreasing $[0,\delta]$, and we therefore have that $y(t)< - \sqrt{8}\nu$ for all time.  Therefore, $\dot{y}< - \frac{1}{8}y^2$ for all time.  Integrating on $[0,t]$, we find
\begin{align*}
 y(t) < (\tfrac{t}{8}+(y(0))^{-1})^{-1}.
\end{align*}
Thus, setting $M=- \sqrt{8}\nu$ and $T^*=-8(y(0))^{-1}=-8\pnt{\int_0^{\pi} \omega_0 \vphi\,dx}^{-1}>0$, we obtain \eqref{blow_up_omega}.
\end{proof}

We notice that in eliminating the advection term $u\omega_x$ from \eqref{Burgers_x}, we have
gone from a non-local equation to a local equation, since $u$ is the
anti-derivative of $\omega$, which is non-local in $\omega$.  Thus, perhaps the
regularizing effect of the term $u\omega_x$ is due in part to its non-local nature.

An analogous effect appears in the context of the Navier-Stokes
equations.  Notice that dropping the advection term in \eqref{NSE_vor} will
affect the pressure term.  Indeed, if we make a drastic alteration to the Navier-Stokes equations and formally drop the pressure term entirely (and also drop \eqref{BC_Paper_NSE_div}, so that the system is not overdetermined), the result is the 3D viscous Burgers equation, which is known to be globally well-posed (see, e.g., \cite{Ladyzhenskaya_1968}). On the other hand, consider the 2D Euler equations.  These equations are known to be globally well-posed, but by formally dropping the pressure term (and again the divergence-free condition), we arrive at the 2D inviscid Burgers equation, which blows up in finite time.

In conclusion any non-local alteration of the nonlinearity in the hydrodynamic equations is in effect leading to alteration in the pressure representation. This in turn might be a stabilizing or destabilizing mechanism of the modified equation.

\section*{Acknowledgements}

This work is supported in part by  the Minerva Stiftung Foundation, and the National Science Foundation grants numbers DMS--1009950, DMS--1109640 and DMS--1109645.

\begin{scriptsize}
\def\cprime{$'$}

\end{scriptsize}

\end{document}